\documentclass{amsart}

\usepackage{amsmath,amssymb,amsthm}

\newtheorem{prop}{Proposition}

\newtheorem{lem}[prop]{Lemma}
\newtheorem{cor}[prop]{Corollary}

\theoremstyle{definition}

\newtheorem*{rem}{Remark}

\newtheorem*{ack}{Acknowledgements}

\def\co{\colon\thinspace}

\newcommand{\calB}{\mathcal{B}}
\newcommand{\oB}{\overline{B}}

\newcommand{\C}{\mathbb{C}}

\newcommand{\D}{\mathbb{D}}
\newcommand{\rmd}{\mathrm{d}}

\newcommand{\rme}{\mathrm{e}}
\newcommand{\bbE}{\mathbb{E}}

\newcommand{\bbF}{\mathbb{F}}

\newcommand{\rmi}{\mathrm{i}}

\newcommand{\calM}{\mathcal{M}}

\newcommand{\N}{\mathbb{N}}

\newcommand{\opartial}{\overline{\partial}}

\newcommand{\R}{\mathbb{R}}

\newcommand{\oz}{\overline{z}}

\newcommand{\Ccinfty}{C_{\mathrm{c}}^{\infty}}

\newcommand{\Cckk}{C_{\mathrm{c}}^{k+1}}

\DeclareMathOperator{\coker}{coker}
\DeclareMathOperator{\image}{im}

\begin{document}

\author[H.~Geiges]{Hansj\"org Geiges}
\author[M.~Sa\u{g}lam]{Murat Sa\u{g}lam}
\address{Mathematisches Institut, Universit\"at zu K\"oln,
Weyertal 86--90, 50931 K\"oln, Germany}
\email{geiges@math.uni-koeln.de}
\email{msaglam@math.uni-koeln.de}
\author[K.~Zehmisch]{Kai Zehmisch}
\address{Fakult\"at f\"ur Mathematik, Ruhr-Universit\"at Bochum,
Universit\"atsstra{\ss}e 150, 44780 Bochum, Germany}
\email{kai.zehmisch@rub.de}

\title[Why bootstrapping fails in $C^k$]{Why bootstrapping for
$J$-holomorphic curves fails in $C^k$}

\date{}

\begin{abstract}
We present a simple example for the failure of the Calder\'on--Zygmund
estimate for the $\opartial$-operator when the Sobolev $(k,p)$-norms
are replaced by the $C^k$-norms. This example is discussed in the
context of elliptic bootstrapping, Fredholm theory, and the regularity of
$J$-holomorphic curves.
\end{abstract}



\maketitle


\section{Introduction}
In the theory of elliptic partial differential equations,
with the Laplace equation as the prototype, it is well known
(see \cite{fero20,jost13}, for instance)
that regularity results can be established for solutions lying
in Sobolev or H\"older spaces, with the help of Calder\'on--Zygmund or
Schauder estimates, respectively.

In symplectic topology, moduli spaces of $J$-holomorphic
curves are described as solution sets of a nonlinear
Cauchy--Riemann equation. Since the implicit function theorem
fails in the Fr\'echet space of smooth maps, one needs to work
with a Banach space of maps having lower regularity.
A Calder\'on--Zygmund estimate for the $\opartial$-operator then is
essential for two purposes:
\begin{itemize}
\item[(i)] regularity results for $J$-holomorphic curves, in the sense
that solutions of the nonlinear Cauchy--Riemann equation
in the Sobolev space $W^{1,p}$ actually turn out to be of class~$C^{\infty}$;
\item[(ii)] existence of $C^k$-bounds for all~$k$ that guarantee compactness
of the relevant moduli space of $J$-holomorphic curves.
\end{itemize}
The Calder\'on--Zygmund estimate allows one to bootstrap from
$W^{k,p}$ to $W^{k+1,p}$; smoothness and $C^k$-bounds (for $p>2$)
then follow from the Sobolev embedding theorem and the corresponding
Sobolev inequality.

Roughly speaking, the estimates say that a (weak) solution $u\in W^{k,p}$
of the inhomogeneous Laplace equation $\Delta u=f$ is two derivatives
more regular than~$f$. The literature abounds with examples
that this statement fails in the smooth theory, that is, from
$\Delta u$, understood in the distributional sense, of class $C^k$
one cannot, in general, conclude that $u\in C^{k+2}$.
This means that a $C^k$-analogue of these estimates
cannot be formulated in a sensible way, because $\|u\|_{C^{k+2}}$ may not
be defined.

However, we have not seen it emphasised that even when the correct
order of differentiability is assumed \emph{a priori}, the Calder\'on--Zygmund
(or Schauder) estimates fail when the Sobolev (or H\"older) norms are
replaced by $C^k$-norms. We allow that this may be fairly apparent to the
more analytically inclined.

In this expository note we adapt an example of Sikorav~\cite{siko94}
(where the $C^k$-esti\-mate cannot be formulated)
to define an explicit family of solutions of the inhomogeneous
Cauchy--Riemann equation, having the correct regularity (in the $C^k$-theory),
but violating the Calder\'on--Zygmund estimate.
We also place this in the context of the Fredholm property of the
Cauchy--Riemann operator, which is essential for showing that the relevant
moduli space of $J$-holomorphic curves is a smooth manifold.

As we shall explain, in the theory of $J$-holomorphic curves one often
deals with these estimates in a setting where the maps are known to be
smooth. As a consequence, the consideration of Sobolev norms on such maps,
or the introduction of Sobolev spaces
of $J$-holomorphic curves, may seem to lack motivation. 
Our example clarifies why one has to work with Sobolev completions.
\section{The Calder\'on--Zygmund estimate}
In this section we formulate the Calder\'on--Zygmund estimate
for the inhomogeneous Cauchy--Riemann equation and briefly discuss
its relevance for the bootstrapping of $J$-holomorphic curves.

Let $B_R\subset\C$ be the open disc of radius~$R$ centred at~$0$.
We write $\Ccinfty(B_R,\C^n)$ for the space of compactly supported
smooth maps $B_R\rightarrow\C^n$, and
$W_0^{k+1,p}(B_R,\C^n)$ for its closure
in the Sobolev space of $k+1$ times weakly differentiable
maps of finite Sobolev $(k+1,p)$-norm. The Cauchy--Riemann operator
is
\[ \opartial:=\partial_{\oz}:=
\frac{1}{2}\bigl(\partial_x+\rmi\partial_y\bigr). \]
Likewise, we are going to set $\partial:=\frac{1}{2}
(\partial_x-\rmi\partial_y)$.
For a proof of the following estimate see~\cite{geze23}.

\begin{prop}
\label{prop:apriori}
For any $k\in\N_0$ and real numbers $p>1$ and $R>0$, there is
a positive constant $c=c(k,p,R)$ such that
\[ \|u\|_{k+1,p}\leq c\|\opartial u\|_{k,p}\;\;\;\text{for all}
\;\;\; u\in W_0^{k+1,p}(B_R,\C^n).\]
\end{prop}

Without the assumption on compact support, one needs to
add the term $c\|u\|_{k,p}$ on the right-hand side, as can be seen
by a partition of unity argument. This is the more common formulation
of the Calder\'on--Zygmund estimate (and sometimes referred to as
a semi-Fredholm estimate, cf.~\cite{geze23}).

One first has to prove the proposition for $u\in\Ccinfty(B_R,\C^n)$,
the stated version then follows by writing $u\in W_0^{k+1,p}(B_R,\C^n)$
as a limit of compactly supported smooth maps.

\begin{rem}
For $p=2$, the proof of Proposition~\ref{prop:apriori} simplifies considerably,
see \cite[Section 4.2]{abho19} or \cite[Section III.1.2]{geze23},
but for the subsequent application of the
Sobolev embedding theorem one needs $p>2$.
\end{rem}

We are going to show by an example that there is no such uniform estimate
$\|u\|_{C^{k+1}}\leq c\|\opartial u\|_{C^k}$.

\begin{prop}
\label{prop:sequence}
For any $k\in\N_0$ there is a sequence $(u_{\nu})$
in $\Cckk(B_{1/2},\C)$ with $\|\opartial u_{\nu}\|_{C^k}$ bounded
uniformly in~$\nu$,
but $\|u_{\nu}\|_{C^{k+1}}\rightarrow\infty$ as $\nu\rightarrow\infty$.
\end{prop}

From Proposition~\ref{prop:apriori}, in \cite{geze23} the regularity of
$J$-holomorphic discs (with Lagrangian boundary condition) is established
by a localisation argument and the difference quotient technique of Abbas and
Hofer~\cite{abho19}. That difference quotient technique allows one to
bootstrap from $W^{k,p}$ to $W^{k+1,p}$, but not from $C^k$ to
$C^{k+1}$, so a Calder\'on--Zygmund estimate in $C^k$ would not be
of help.

A different approach to the regularity of $J$-holomorphic curves
can be found in \cite[Section~B.4]{mcsa12}. Here the nonlinear
Cauchy--Riemann equation $u_x+J(u)u_y=0$ is reformulated as
an inhomogeneous linear equation, and then one directly uses the
regularity theory for the $\opartial$-operator. This approach would
stumble at the first hurdle in the $C^k$-theory by Sikorav's example.

Much of the compactness theory of $J$-holomorphic curves as in
\cite{geze23} would go through in the smooth theory if one had
a Calder\'on--Zygmund estimate (for smooth maps) in the $C^k$-norms.
Our example shows why this hope is in vain.
\section{The example}
\label{section:example}
\subsection{Sikorav's example}
\label{subsection:Sikorav}
We begin with an example of a function $f\co B_{1/2}\rightarrow\C$
that is not of class $C^1$, even though $\opartial f$ (in the
distributional sense) is of
class~$C^0$. This is a slight modification (and correction)
of an example presented by Sikorav~\cite{siko94}, which is closely
related to the standard example illustrating the corresponding
phenomenon for the Laplace operator, see~\cite{fero20}.
Set
\begin{equation}
\label{eqn:Sikorav}
f(z)=\begin{cases}
z\log\log|z|^{-2} & \text{for $z\in B_{1/2}\setminus\{0\}$},\\
0                 & \text{for $z=0$}.
\end{cases}
\end{equation}
Then
\[ \opartial f=-\frac{z}{\oz\log|z|^{-2}}\;\;\;\text{for $z\neq 0$}.\]
This extends continuously (with value~$0$) into~$z=0$, and this continuous
extension is the distributional derivative~$\opartial f$ on $B_{1/2}$
(see the discussion in Section~\ref{section:Fredholm}).

On the other hand, we have
\[ \partial f=\log\log|z|^{-2}-\frac{1}{\log|z|^{-2}}\;\;\;
\text{for $z\neq 0$},\]
which does not extend continuously into~$z=0$.
\subsection{Proof of Proposition~\ref{prop:sequence}}
\label{subsection:proof}
Here is the example for the failure of the Calder\'on--Zygmund estimate
in the $C^k$-theory.
Choose a smooth function $\psi\co\R_0^+\rightarrow[0,1]$ compactly supported
in $[0,1/4)$ and with $\psi\equiv 1$ on $[0,1/16]$. For $\nu\in\N$ we define
$f_{\nu}\co B_{1/2}\rightarrow\C$ by
\[ f_{\nu}(z)=\begin{cases}
z|z|^{1/\nu}\log\log|z|^{-2}
   & \text{for $z\in B_{1/2}\setminus\{0\}$},\\
0                            
   & \text{for $z=0$},
\end{cases}\]
and we set $u_{\nu}(z)=\psi(|z|^2)\cdot f_{\nu}(z)$.

Since $|z|^{1/\nu}\log\log|z|^{-2}\rightarrow 0$ as $z\rightarrow 0$,
the function $f_{\nu}$ is complex differentiable in $z=0$
with $\partial f_{\nu}(0)$ equal to~$0$, and hence
differentiable with $\opartial f_{\nu}(0)$ likewise equal to~$0$.

Writing $|z|^{1/\nu}$ as $(z\oz)^{1/2\nu}$ we see that
\[ \opartial |z|^{1/\nu}=\frac{1}{2\nu}|z|^{\frac{1}{\nu}-2}z
\;\;\;\text{for $z\neq 0$}.\]
We then compute
\begin{equation}
\label{eqn:opartialf}
\opartial f_{\nu}(z)=
\frac{1}{2\nu}|z|^{\frac{1}{\nu}-2}z^2\log\log|z|^{-2}
-\frac{z|z|^{1/\nu}}{\oz\log|z|^{-2}}\;\;\;
\text{for $z\neq 0$},
\end{equation}
and
\[ \partial f_{\nu}(z)=
|z|^{1/\nu}\log\log|z|^{-2}+\frac{\oz}{z}\,\opartial f_{\nu}(z)
\;\;\;\text{for $z\neq 0$}.\]
Since both $\opartial f_{\nu}(z)$ and $\partial f_{\nu}(z)$
go to $0$ as $z\rightarrow 0$, we conclude that $f_{\nu}$ (and hence
$u_{\nu}$) is of class~$C^1$.

The second summand in $\opartial f_{\nu}$
is bounded uniformly in~$\nu$ on $B_{1/2}\setminus\{0\}$.
Writing $|z|=r^{\nu}$ for $z\neq 0$ with $0<r<2^{-1/\nu}<1$,
we see that the first summand in $\opartial f_{\nu}$
is likewise bounded uniformly in~$\nu$, since
\begin{equation}
\label{eqn:r}
0<\frac{1}{2\nu}r\log\log r^{-2\nu}<
\frac{1}{2\nu}r\log r^{-2\nu}=-r\log r,
\end{equation}
which extends continuously into $r=0$.
Clearly, these bounds also take care of the second summand in
$\partial f_{\nu}$.

On the other hand, the first summand of $\partial f_{\nu}(z)$
evaluated at $z=2^{-\nu}$ yields
$\frac{1}{2}\log\log 2^{2\nu}$,
which goes to infinity as $\nu\rightarrow\infty$.
It follows that $\|\opartial f_{\nu}\|_{C^0}$ is bounded uniformly in~$\nu$,
whereas $\|f_{\nu}\|_{C^1}$ goes to infinity as $\nu\rightarrow\infty$.

The same is true for the compactly supported functions $u_{\nu}$, since
$\|f_{\nu}\|_{C^0}$ is bounded uniformly in $\nu$, and
\[ \opartial u_{\nu}(z)=
\psi'(|z|^2)z f_{\nu}(z)+\psi(|z|^2)\opartial f_{\nu}(z),\]
with a similar expression for $\partial u_{\nu}$, which means that
the limiting behaviour of $\opartial u_{\nu}$
and $\partial u_{\nu}$ equals that of $\opartial f_{\nu}$ and
$\partial f_{\nu}$, respectively.

In order to get examples for the higher $C^k$-norms, simply start from
the definition $f_{\nu}(z)=z^{k+1}|z|^{1/\nu}\log\log|z|^{-2}$
for $z\in B_{1/2}\setminus\{0\}$, for any $k\in\N_0$.
\section{The Fredholm property of $\opartial$}
\label{section:Fredholm}
The discussion so far shows that one cannot forgo Sobolev norms
for bootstrapping arguments, but it still seems to leave room
for the possibility to stay within the framework of
$C^k$-maps. In the compactness theory of $J$-holomorphic
curves one cannot simply work in a space of $C^k$-maps for
some fixed~$k$, since one typically relies on the Arzel\`a--Ascoli theorem
and $C^{k+1}$-bounds to guarantee convergence in~$C^k$.
However, one might want to start with $J$-holomorphic curves
of class~$C^1$, interpret them as maps of class $W^{1,p}$,
and then use elliptic bootstrapping (with respect to
Sobolev norms) and Sobolev embedding to show
that the curves are in fact smooth.
\subsection{Calder\'on--Zygmund estimate and Fredholm property}
As we want to explain now, it is not possible to avoid altogether the
use of Sobolev spaces of maps. Typically, for a given geometric
problem in symplectic topology, one describes a moduli space $\calM$
of $J$-holomorphic curves in the form
\[ \calM=\{u\in\calB\co \opartial_Ju=0\},\]
where $\calB$ is a Banach space of maps $u\co\Sigma\rightarrow (M,J)$
from a compact Riemann surface $\Sigma$ into an almost complex manifold
$(M,J)$ (with $J$ tamed by some symplectic form~$\omega$), subject to
(e.g.\ Lagrangian) boundary conditions when $\partial\Sigma\neq\emptyset$.
The nonlinear Cauchy--Riemann equation $\opartial_Ju:=u_x+J(u)u_y=0$
describes the $J$-holomorphicity of the map~$u$.

In order to establish that $\calM$ is a manifold of the expected dimension,
one needs to verify that $\opartial_J$ is a Fredholm operator, so that
one can apply the theorem of Sard--Smale. By a perturbation argument
it may suffice to do this for the linear Cauchy--Riemann operator
$\opartial$. This Fredholm property,
as we shall see presently, holds for Sobolev spaces,
but it is violated in the $C^k$-realm.

Consider a bounded linear operator $T\co\bbE\rightarrow\bbF$
between Banach spaces with $\dim\ker T<\infty$. Let
$\bbE_1$ be a closed complement of $\ker T$ in~$\bbE$.
It then follows from the open mapping theorem,
applied to $T|_{\bbE_1}\co\bbE_1\rightarrow T(\bbE)$, that
$T$ has a closed image if and only if we have an estimate $\|x\|_{\bbE}\leq 
c\|Tx\|_{\bbF}$ for all $x\in\bbE_1$. The image $T(\bbE)$ being closed is
a necessary condition for $\coker T$ to be finite, i.e.\
the Fredholm property of~$T$.

Thus, whether or not $\opartial$ has a closed image
(when regarded as an operator between certain Banach spaces of functions)
is equivalent to the existence or the failure
of the Calder\'on--Zygmund estimate in the corresponding norms, provided
$\ker\opartial$ is finite-dimensional in the given setting.

Here is an example how to use Proposition~\ref{prop:sequence}
to show the failure of the Fredholm property in the
$C^k$-theory of $J$-holomorphic discs.
Write $\D\subset\C$ for the closed unit disc, and $C^1_{\R}(\D,\C)$
for the space of $C^1$-maps $\D\rightarrow\C$ with real boundary values.
We may regard $C^1_{\mathrm{c}}(B_{1/2},\C)$ as
a subspace of $C^1_{\R}(\D,\C)$, and $C^0_{\mathrm{c}}(B_{1/2},\C)$
as a subspace of $C^0(\D,\C)$. Notice that the ambient spaces
are Banach spaces, but the subspaces are not closed.

\begin{cor}
\label{cor:not-F}
The operator $\opartial\co C^1_{\R}(\D,\C)\rightarrow
C^0(\D,\C)$ is not Fredholm.
\end{cor}

\begin{proof}
A function in $\ker\opartial$
can be extended by Schwarz reflection in the unit circle to a bounded
holomorphic function on~$\C$, which is constant by Liouville's theorem.
This implies $\ker\opartial=\R$. Moreover, by
Proposition~\ref{prop:sequence}, the operator $\opartial$
violates the Calder\'on--Zygmund estimate.
Thus, $\image\opartial$ is not closed.
\end{proof}
\subsection{Failure of the Fredholm property in $C^k$}
We now want to use Sikorav's example to demonstrate by a specific example
that the image of $\opartial\co C^1_{\R}(\D,\C)\rightarrow C^0(\D,\C)$
is not closed, and thus give a more concrete proof of
Corollary~\ref{cor:not-F}.

\begin{prop}
\label{prop:not-closed}
The weak derivative $\opartial(\psi f)\in C^0(\D,\C)$ --- with $f$ as
in~\eqref{eqn:Sikorav},
and $\psi$ the cut-off function from Section~\ref{subsection:proof} ---
is in the closure of $\image\opartial$, but not
itself an element of that image.
\end{prop}

We first present a `classical' argument using mollification, and then
an alternative approach using the sequence $(f_{\nu})$ introduced
in Section~\ref{subsection:proof}.
\subsection{Proof by mollification}
We begin by analysing Sikorav's example a little more carefully.

\begin{lem}
The function $f$ defined in \eqref{eqn:Sikorav} is
an element of $W^{1,p}(B_{1/2},\C)$ for any $p\in[1,\infty)$,
and the weak derivatives $\opartial f$ and
$\partial f$ may be assumed to coincide with the actual
derivatives on $B_{1/2}\setminus\{0\}$.
\end{lem}

\begin{proof}
(i) First we are going to show that the weak $\partial$- and
$\opartial$-derivatives of $f$ are as claimed.
We consider $\partial f$; for $\opartial f$ the argument is
completely analogous. For $\varepsilon>0$, let $\chi_{\varepsilon}\in
C^{\infty}(\C)$ be a cut-off function
with $\chi_{\varepsilon}\equiv 0$ on $B_{\varepsilon}$, and
$\chi_{\varepsilon}\equiv 1$ outside $B_{2\varepsilon}$.
Set $f_{\varepsilon}=\chi_{\varepsilon}f\in C^{\infty}(B_{1/2})$.
For any test function $\varphi\in C^{\infty}_{\mathrm{c}}(B_{1/2})$,
integration by parts gives
\begin{equation}
\label{eqn:int-parts}
\int_{B_{1/2}}(\partial f_{\varepsilon})\varphi=
-\int_{B_{1/2}}f_{\varepsilon}(\partial\varphi).
\end{equation}

We may assume that
\[ |\partial\chi_{\varepsilon}|\leq\frac{c}{2\varepsilon}\]
for some constant~$c$.
For $|z|\leq\varepsilon$ and $|z|\geq 2\varepsilon$, the derivative
$\partial\chi_{\varepsilon}$ vanishes identically. It follows that
\[ |\partial\chi_{\varepsilon}(z)|\leq\frac{c}{|z|}\;\;\;
\text{for $z\in B_{1/2}\setminus\{0\}$.}\]
From
\[ \partial f_{\varepsilon}=\chi_{\varepsilon}(\partial f)+
(\partial\chi_{\varepsilon})f\]
we then conclude that
\[ |\partial f_{\varepsilon}(z)|\leq|\partial f(z)|
+c\left|\frac{f(z)}{z}\right|\;\;\;\text{for $z\in B_{1/2}\setminus\{0\}$.}\]
Now, $\partial f_{\varepsilon}$
converges pointwise on $B_{1/2}\setminus\{0\}$ to $\partial f$.
Hence, provided the functions $f/z$ and $\partial f$ are
integrable, we can take the limit $\varepsilon\searrow 0$
in \eqref{eqn:int-parts} and conclude with the
Lebesgue dominated convergence theorem that
\[ \int_{B_{1/2}}(\partial f)\varphi = -\int_{B_{1/2}}f(\partial\varphi),\]
so $\partial f$ constitutes the weak $\partial$-derivative of~$f$.

(ii) It remains to show that the functions $f/z$, $\partial f$ and
$\opartial f$ are in $L^p(B_{1/2},\C)$ for any $p\in [1,\infty)$.
Both $\opartial f$ and the function $z\mapsto 1/\log|z|^{-2}$
extend continuously to $\oB_{1/2}$, so we need only show that
\[ z\longmapsto \log\log|z|^{-2},\;\;\; z\neq 0,\]
is in $L^p(B_{1/2})$.

For $r\in (0,1/2)$ we have
\[ 0<\log\log r^{-2}<\log r^{-2}=2\log r^{-1},\]
so it suffices to show that $r\mapsto\log r^{-1}$
is an $L^p$-function on the interval $(0,1/2)$.
In fact, this function is even $L^p$-integrable on
$(0,1)$, as can be seen by the substitution
$t=\log r^{-1}$, $t\in(0,\infty)$, which yields a
transformation to the $\Gamma$-function. For with
$r=\rme^{-t}$ and $\rmd r=-\rme^{-t}\,\rmd t$ we have
\[ \int_0^1\bigl(\log r^{-1}\bigr)^p\,\rmd r=
\int_0^{\infty}t^p\rme^{-t}\,\rmd t=\Gamma(p+1).\]
We conclude that the function $z\mapsto\log\log|z|^{-2}$ is in
$L^p(B_{1/2},\C)$.
\end{proof}

For the basic theory of mollifiers we use presently
in the first proof of
Proposition~\ref{prop:not-closed}, see \cite[Sections C.5 and 5.3]{evan10}.

\begin{proof}[First proof of Proposition~\ref{prop:not-closed}]
Let $\rho\in C^{\infty}(\C)$ be the standard mollifier,
\[ \rho(z)=\begin{cases}
C\cdot\exp\Bigl(\frac{1}{|z|^2-1}\Bigr) & \text{for $|z|<1$},\\
0                                      & \text{for $|z|\geq 1$},
\end{cases}\]
with $C\in\R^+$ chosen such that $\int_{\R^2}\rho=1$.
Set $\rho_{\varepsilon}(z)=\frac{1}{\varepsilon^2}\rho
\bigl(\frac{z}{\varepsilon}\bigr)$.

The function $f$ can be defined on $B_{0.6}$, so for $\varepsilon<0.1$
we can define the mollification
$f^{\varepsilon}=\rho_{\varepsilon}*f\in C^{\infty}(\oB_{1/2})$ of~$f$,
that is,
\[ f^{\varepsilon}(z):=\int_{B_{0.6}}\rho_{\varepsilon}(z-w)
f(w)\,\rmd\lambda^2_w=\int_{B_{\varepsilon}}\rho_{\varepsilon}(w)
f(z-w)\,\rmd\lambda^2_w\;\;\;\text{for $z\in\oB_{1/2}$},\]
where $\rmd\lambda^2_w$ denotes the $2$-dimensional Lebesgue measure
with respect to the variable~$w$.
As $\varepsilon\searrow 0$, the function $f^{\varepsilon}$ converges
to $f$ in $W^{1,p}(B_{1/2},\C)$, and for $p>2$ this convergence
is uniform in $C^0(\oB_{1/2},\C)$ by the Sobolev embedding theorem.

Since $\opartial f$ is continuous on $\oB_{1/2}$, we have
\[ \opartial f^{\varepsilon}=\opartial(\rho_{\varepsilon}* f)
=\rho_{\varepsilon}*\opartial f\longrightarrow\opartial f\]
uniformly on $\oB_{1/2}$ as $\varepsilon\searrow 0$.

Now let $\psi$ be the cut-off function defined in
Section~\ref{subsection:proof}, and set
\[ u^{\varepsilon}(z)=\psi(|z|^2)f^{\varepsilon}(z).\]
This function is compactly supported in $B_{1/2}$, and we may regard it
as an element of $C^1_{\R}(\D,\C)$. Then
\[ \opartial u^{\varepsilon}(z)=\psi'(|z|^2)zf^{\varepsilon}(z)+
\psi(|z|^2)\opartial f^{\varepsilon}(z)
\longrightarrow
\psi'(|z|^2)zf(z)+\psi(|z|^2)\opartial f(z)\]
uniformly on $\D$, and this limit equals
the weak $\opartial$-derivative of $\psi f$.

But $\opartial(\psi f)$ --- regarded as an element of
$C^0(\D,\C)$ --- does not equal $\opartial h$
for any $h\in C^1_{\R}(\D,\C)$, for otherwise
we would have $\opartial (\psi f-h)=0$ (in the
distributional sense), which by the regularity of
the $\opartial$-operator would entail that
$\psi f-h$ is of class $C^{\infty}$,
contradicting the fact that $\psi f$ is not
of class~$C^1$. Indeed, any weak solution $u\in L^1_{\mathrm{loc}}$
of the equation $\opartial u=0$ is also a weak solution of
the Laplace equation $\Delta u=0$, and hence harmonic
(and, in particular, smooth) by Weyl's lemma.
\end{proof}
\subsection{Proof using the sequence $(f_{\nu})$}
We now give a more explicit construction of a sequence
in $\image\opartial\subset C^0(\D,\C)$
with limit not contained in that image, based directly on the
sequence $(f_{\nu})$ presented in Section~\ref{subsection:proof}.

The idea is to interpolate between $f_{\nu}$
(or rather $4f_{\nu}$) near $z=0$ and
$f$ outside a neighbourhood of $z=0$ that is shrinking as
$\nu\rightarrow\infty$. To do so, consider the continuous function
\begin{equation}
\label{eqn:cont-fn}
t\longmapsto\begin{cases}
4t^{\frac{1}{2\nu}} & \text{for $0\leq t\leq 16^{-\nu}$},\\
1                   & \text{for $16^{-\nu}\leq t<1/2$},
\end{cases}
\end{equation}
which is smooth away from $t=0$ and $t=16^{-\nu}$.
Let $\phi_{\nu}\co
[0,1/2)\rightarrow [0,1]$ be a smoothening of this function at
$t=16^{-\nu}$. Specifically, $\phi_{\nu}$ is a
function that coincides with \eqref{eqn:cont-fn}
outside the interval $\bigl[\frac{1}{16^{\nu}+1},
\frac{1}{16^{\nu}}\bigr]$, and such that $\phi_{\nu}'$ is pointwise at most
double the slope of~\eqref{eqn:cont-fn}, that is,
\begin{equation}
\label{eqn:phi-nu-prime}
0\leq\phi_{\nu}'<\frac{4}{\nu}t^{\frac{1}{2\nu}-1}\;\;\;
\text{for $t\in(0,16^{-\nu})$},
\end{equation}
and $\phi_{\nu}'(t)=0$ for $t\geq 16^{-\nu}$.

Now set $g_{\nu}(z)=\phi_{\nu}(|z|^2)f(z)$. This function coincides
with $f$ for $|z|\geq 4^{-\nu}$, and with $4f_{\nu}$ near $z=0$.
In particular, it is of class~$C^1$.
We have
\[ \opartial g_{\nu}(z)=\phi_{\nu}'(|z|^2)zf(z)+\phi_{\nu}(|z|^2)\opartial f(z)
\;\;\;\text{for $z\in B_{1/2}\setminus\{0\}$},\]
and $\opartial g_{\nu}(0)=4\opartial f_{\nu}(0)=0=\opartial f(0)$.

We claim that $\opartial g_{\nu}\rightarrow\opartial f$ uniformly on~$B_{1/2}$.
To this end we estimate
\begin{equation}
\label{eqn:triangle}
|\opartial g_{\nu}(z)-\opartial f(z)|
\leq |\phi_{\nu}'(|z|^2)zf(z)|+|(\phi_{\nu}(|z|^2)-1)\opartial f(z)|.
\end{equation}
Thanks to \eqref{eqn:phi-nu-prime}, the first summand
is at most equal to $8$ times the norm of the first summand
in~\eqref{eqn:opartialf}. With \eqref{eqn:r} and the observation
that the function $r\mapsto -r\log r$ is monotone increasing
on the interval $(0,\rme^{-1})$, we see that the first summand on
the right-hand side of \eqref{eqn:triangle} is bounded uniformly by
$-8\cdot 16^{-\nu}\log 16^{-\nu}$,
which goes to $0$ as $\nu\rightarrow\infty$.

Regarding the second summand in \eqref{eqn:triangle}, notice that
$|\phi_{\nu}(|z|^2)-1|$ is bounded by~$1$, and it is identically equal to
$0$ for $z\geq 16^{-\nu}$. It follows that
\[ |(\phi_{\nu}(|z|^2)-1)\opartial f(z)|\leq
\max_{|\zeta|\leq 16^{-\nu}}|\opartial f(\zeta)|\;\;\;
\text{for $z\in B_{1/2}$}.\]
Since $\opartial f$ is continuous with $\opartial f(0)=0$, this
bound on the right-hand side
goes to $0$ as $\nu\rightarrow\infty$. This proves the claim.

Similarly, $\opartial(\psi  g_{\nu})\rightarrow\opartial
(\psi f)$ uniformly in $C^0(\D,\C)$.
This second proof of Proposition~\ref{prop:not-closed}
now concludes just like the first one.
\begin{ack}
We thank Alberto Abbondandolo, Peter Albers and Guido Sweers for useful
conversations. The authors are partially supported by the
Sonderforschungsbereich TRR 191
\textit{Symplectic Structures in Geometry, Algebra and Dynamics},
funded by the DFG (Projektnummer 281071066 -- TRR 191).
\end{ack}


\begin{thebibliography}{10}
%
\bibitem{abho19}
\textsc{C. Abbas and H. Hofer},
\textit{Holomorphic Curves and Global Questions in Contact Geometry},
Basler Lehrb\"ucher, Birkh\"auser Verlag, Basel (2019).
%
\bibitem{aula94}
\textsc{M. Audin and J. Lafontaine} (eds.),
\textit{Holomorphic Curves in Symplectic Geometry},
Progress in Mathematics \textbf{117},
Birkh\"auser Verlag, Basel (1994).
%
\bibitem{evan10}
\textsc{L. C. Evans},
\textit{Partial Differential Equations} (2nd edn),
Graduate Studies in Mathematics \textbf{19},
American Mathematical Society,
Providence, RI (2010).
%
\bibitem{fero20}
\textsc{X. Fern\'andez Real and X. Ros Oton},
\textit{Regularity Theory for Elliptic PDE},
lecture notes, Z\"urich (2020).
%
\bibitem{geze23}
\textsc{H. Geiges and K. Zehmisch},
\textit{A Course on Holomorphic Discs},
book in preparation.
%
\bibitem{jost13}
\textsc{J. Jost},
\textit{Partial Differential Equations} (3rd edn),
Graduate Texts in Mathematics \textbf{214},
Springer-Verlag, New York (2013).
%
\bibitem{mcdu90}
\textsc{D. McDuff},
Elliptic methods in symplectic geometry,
\textit{Bull. Amer. Math. Soc. (N.S.)}
\textbf{23} (1990), 311--358.
%
\bibitem{mcsa12}
\textsc{D. McDuff and D. Salamon},
\textit{$J$-holomorphic Curves and Symplectic Topology} (2nd edn),
American Mathematical Society Colloquium Publications \textbf{52},
American Mathematical Society, Providence, RI (2012).
%
\bibitem{siko94}
\textsc{J.-C. Sikorav},
Some properties of holomorphic curves in almost complex manifolds,
in \cite{aula94}, 165--189.
%
\end{thebibliography}
\end{document}